\documentclass[11pt]{amsart}

\usepackage{amssymb, mathrsfs, amsmath, tikz-cd,amsthm,} 
\usepackage[margin=1in]{geometry} 

\usepackage[colorlinks=true, citecolor=blue, linkcolor=blue]{hyperref}

\usepackage{float}
\usepackage{enumerate}
\usepackage{comment}
\usepackage{tikz, braids}
\usepackage{enumerate}
\usepackage{comment}
\usepackage{caption,subcaption}
\usetikzlibrary{positioning}
\usepackage[ugly]{nicefrac}
\usepackage{xparse, braket}

\newcommand{\ints}{\mathbb{Z}}

\newcommand{\inv}{^{-1}}

\renewcommand{\tilde}{\widetilde}
\DeclareMathOperator{\Hom}{Hom}

\DeclareMathOperator{\C}{\mathbb{C}}

\DeclareMathOperator{\level}{level}
\DeclareMathOperator{\glevel}{geolevel}
\DeclareMathOperator{\po}{po}
\DeclareMathOperator{\spec}{spec}

\DeclareMathOperator{\Mod}{Mod}
\DeclareMathOperator{\Rep}{Rep}
\newcommand{\normcl}[1]{\langle \langle #1 \rangle \rangle}

\newtheorem*{corollary*}{Corollary}
\newtheorem*{lemma*}{Lemma}

\newtheorem*{theorem*}{Theorem}

\numberwithin{equation}{section}

\newtheorem{theorem}{Theorem}[section]
\newtheorem{definition}[theorem]{Definition}
\newtheorem{corollary}[theorem]{Corollary}
\newtheorem{lemma}[theorem]{Lemma}
\newtheorem{proposition}[theorem]{Proposition}
\newtheorem{observation}[theorem]{Observation}

\newtheorem{innercustomthm}{Theorem}
\newenvironment{customthm}[1]
  {\renewcommand\theinnercustomthm{#1}\innercustomthm}
  {\endinnercustomthm}

\begin{document}

\title[Congruence subgroups and $B_3$ representations ]{Congruence subgroups from representations of the three-strand  braid group}

\author{Joseph Ricci}
\email{ricci@math.ucsb.edu}
\address{Dept. of Mathematics\\
    University of California\\
    Santa Barbara, CA 93106-6105\\
    U.S.A.}

\author{Zhenghan Wang}
\email{zhenghwa@microsoft.com}
\address{Microsoft Station Q and Dept. of Mathematics\\
    University of California\\
    Santa Barbara, CA 93106-6105\\
    U.S.A.}

\thanks{The second author is partially supported by NSF grants DMS-1410144 and DMS-1411212.}

\date{\today}

\begin{abstract} Ng and Schauenburg proved that the kernel of a $(2+1)$-dimensional topological quantum field theory representation of $\mathrm{SL}(2, \ints)$ is a congruence subgroup. Motivated by their result, we explore when the kernel of an irreducible representation of the braid group $B_3$ with finite image enjoys a congruence subgroup property. In particular, we show that in dimensions two and three, when the projective order of the image of the braid generator $\sigma_1$ is between 2 and 5 the kernel projects onto a congruence subgroup of $\mathrm{PSL}(2,\ints)$ and compute its level. However, we prove for three dimensional representations, the projective order is not enough to decide the congruence property. For each integer of the form $2\ell \geq 6$ with $\ell$ odd, we construct a pair of non-congruence subgroups associated with three-dimensional representations having finite image and $\sigma_1$ mapping to a matrix with projective order $2\ell$. Our technique uses classification results of low dimensional braid group representations, and the Fricke-Wohlfarht theorem in number theory.

\end{abstract}

\maketitle

\section{Introduction}

The double cover $\mathrm{SL}(2,\ints)$ of the modular group $\mathrm{PSL}(2,\ints)$ naturally occurs in quantum topology as the mapping class group of the torus.  Let $\Sigma_{g,n}$ be the orientable genus $g$ surface with $n$ punctures and denote by $\Mod(\Sigma_{g,n})$ its mapping class group. A (2+1)-dimensional topological quantum field theory (TQFT) affords a projective representation of $\Mod(\Sigma_{g,n})$ which we refer to as a quantum representation.  An amazing theorem of
Ng and Schauenburg \cite{NgSch} says that the kernel of the quantum representations of  $\mathrm{SL}(2,\ints)$ is always a congruence subgroup.  The modular group is also disguised as the three-strand braid group $B_3$ through the central extension: $1\rightarrow \mathbb{Z}=\langle(\sigma_1\sigma_2)^3 \rangle \rightarrow B_3\rightarrow \mathrm{PSL}(2,\ints) \rightarrow 1$.  Each simple object of the modular tensor category $\mathcal{C}$ associated to a $(2+1)$-TQFT gives rise to a representation of $B_3$.  Are there versions of the Ng-Schauenburg congruence kernel theorem for those braid group representations?  We initiate a systematical investigation of this problem and find that a native generalization does not hold.

To pass from a representation of $B_3$ to the modular group $\mathrm{PSL}(2,\ints)$, we consider only irreducible representations $\rho_x : B_3\rightarrow \mathrm{GL}(d,\C)$ associated to a simple object $x$ of a modular tensor category $\mathcal{C}$.  Then the generator $(\sigma_1\sigma_2)^3$ of the center of $B_3$ is a scalar of finite order.  By rescaling $\rho_x$ with a root of unity $\xi$, we obtain a representation of the modular group $\rho_{x,\xi} : \mathrm{PSL}(2,\ints) \rightarrow \mathrm{GL}(d,\C)$.
By the property F conjecture, the representations $\rho_{x,\xi}$ should have finite images if the squared quantum dimension $d_x^2$ of $x$ is an integer.  For the Ising anyon $\sigma$, the kernel is indeed a congruence subgroup, but the kernel for the anyon denoted as $G$ in $D(S_3)$ is not \cite{Cui}.  Therefore, when a property F anyon has a congruence subgroup property is more subtle.  In this paper we systematically explore the low dimensional irreducible representations of $B_3$ with finite images, and determine when the kernel is a congruence subgroup.

Congruence subgroups of $\mathrm{SL}(2,\ints)$ are well-studied as they are easy examples of finite index subgroups of $\mathrm{SL}(2,\ints)$ and because of their role in the theory of modular forms and functions. Their relative scarcity in the collection of all finite index subgroups of $\mathrm{SL}(2,\ints)$ makes them of interest. Indeed, if we let $N_c(n)$ (resp. $N(n)$) denote the number of congruence subgroups (resp. subgroups) of $\mathrm{SL}(2,\ints)$ with index $n$ then $N_c(n) / N(n) \to 0$ as $n \to \infty$ {\cite{Sto}}. Contrast this with the result of Bass, Lazard, and Serre \cite{BLS} and separately Mennicke \cite{Men} stating that for $d$ greater than two any finite index subgroup of $\mathrm{SL}(d,\ints)$ is a congruence subgroup.

Another motivation of this research is to study the vector-valued modular forms (VVMF) associated to congruence subgroups (see \cite{Gannon} and the references therein).  VVMFs provide deep insight for the study of TQFTs and conformal field theories (CFTs).  Since the general VVMF theory applies also to non-congruence subgroups, TQFT representations of $B_3$ provide interesting test ground of the theory and conversely, VVMF could provide deep insight into the study of the TQFT representations of $B_3$ even in the non-congruence case. The matrices contained in the image of a $B_3$ representation can also be used as quantum gates for topological quantum computations, the congruence property of $B_3$ representations might even find application to quantum information processing \cite{Wang}.
 
\subsection{Main results} Our two main theorems address the question raised above. For a square matrix $A$ with eigenvalues $\lambda_1, \ldots, \lambda_d$ denote by $\po(A)$ the least positive integer $t$ so that $\lambda_1^t = \cdots = \lambda_d^t$. This is the projective order of $A$. Rowell and Tuba determined in {\cite{RT}} a criteria for deciding (in almost all cases) when the image of an irreducible representation $\rho$ of dimension five or less has finite image. When $\rho$ is two or three-dimensional, this often depends only on the projective order of $\rho(\sigma_1)$ and is therefore determined by the eigenvalues of this matrix. Upon scaling by a character, we can assume such a representation factors through the quotient map $\pi: B_3 \to \mathrm{PSL}(2,\ints)$. Motivated by the result of Ng and Schauenburg, we ask when the kernel of the induced representation of $\mathrm{PSL}(2,\ints)$ is a congruence subgroup.

\begin{customthm}{A} Let $d=2$ or $3$ and suppose $\rho: B_3 \to \mathrm{GL}(d,\C)$ is a $d$-dimensional representation with finite image that factors through $\pi$. If $2 \leq \po(\rho(\sigma_1)) \leq 5$ then $\pi(\ker \rho)$ is a congruence subgroup with level equal to the order of $\rho(\sigma_1)$.
\end{customthm}

We immediately are able to conclude:

\begin{corollary*} Every two-dimensional irreducible quantum representation of $B_3$ with finite image can be scaled so that its kernel projects onto a congruence subgroup of $B_3/Z(B_3)= \mathrm{PSL}(2,\ints)$. Every three-dimensional irreducible quantum representation of $B_3$ with finite image and $2 \leq \po(\rho(\sigma_1)) \leq 5$ can be scaled so that its kernel projects onto a congruence subgroup of $\mathrm{PSL}(2,\ints)$.
\end{corollary*}

The most important tool we will use for the proof of the above is the Fricke-Wohlfarht theorem. This relates the (ordinary) level and the geometric level of a congruence subgroup of $\mathrm{PSL}(2,\ints)$. In particular, if we let $N$ be the order of $\rho(\sigma_1)$ then as a consequence of Fricke-Wohlfarht, we find that $\pi(\ker \rho)$ is a congruence subgroup if and only if the principal congruence subgroup of level $N$ is a subgroup of $\pi(\ker\rho)$. This is Corollary \ref{maincorollary}. To finally obtain our main result we will show that an arbitrary two-dimensional representation factors through $\mathrm{PSL}(2,N)$ with $N$ as above. This can be explicitly checked using the presentation given in \cite{Hsu}. Applying the result of Rowell and Tuba, we find that there are only a small number of cases to check. Moreover, the above property is determined by the eigenvalues of $\rho(\sigma_1)$ which is both pleasing and expected. Our second main result points out that for three dimensional representations, once the projective order of $\rho(\sigma_1)$ is greater than five, the projective order is not enough to determine the congruence properties of the induced kernel.

\begin{customthm}{B} For any positive integer of the form $2\ell >2 $ with $\ell$ odd, there are two representations $\rho_{\ell,\pm}: B_3 \to \mathrm{GL}(3,\C)$ with finite image that factors through $\pi$ for which $\po(\rho_{\ell,\pm}(\sigma_1))=2\ell$ and each $\pi(\ker\rho_{\ell,\pm})$ is not a congruence subgroup.
\end{customthm}

\section{Background}

In this section we will develop the some of the basic material necessary for our main theorem. 

\subsection{Congruence subgroups}
The finite index subgroups of an infinite group hold much information about the group. The congruence subgroups of $\mathrm{SL}(d,\ints)$ and its projectivization $\mathrm{PSL}(d,\ints)$ are one class such class subgroups. In some sense, they are the easy examples of such subgroups. Let us give the formal definition. 

\begin{definition} Let $d$ be greater than one. For each positive integer $N$ there is a short exact sequence 
\begin{equation*}
1 \to \ker {\varphi_N} \to \mathrm{PSL}(d,\ints) \stackrel{\varphi_N}{\to} \mathrm{PSL}(d,\ints /N \ints) \to 1
\end{equation*}
where $\varphi_N$ is the homomorphism given by reducing mod $N$.  The kernel of this homorphism is a finite index subgroup of $\mathrm{PSL}(d,\ints)$, called the {\bf{principal congruence subgroup of $PSL(d,\ints)$ of level $N$}}. Any finite index subgroup subgroup $G$ of $\mathrm{PSL}(d,\ints)$ containing $\ker\varphi_N$ for some $N>1$ is called a {\bf{congruence subgroup}} and otherwise we say $G$ is {\bf{non-congruence}}.
\end{definition}

The above definition brings one to ask whether every finite index subgroup of $\mathrm{PSL}(d,\ints)$ is a congruence subgroup. This was answered in \cite{BLS,Men}. In particular, the question of congruence is interesting only for subgroups of $\mathrm{PSL}(2,\ints)$.

\begin{theorem} Every finite index subgroup of $\mathrm{PSL}(d,\ints)$ is a congruence subgroup if and only if $d$ is greater than two. 
\end{theorem}

\subsection{The modular group}

For the rest of this paper, let $\Gamma$ denote $\mathrm{PSL}(2,\ints)$. It is well-known that $\Gamma$ is isomorphic to the free product $(\ints/2\ints) \ast (\ints/ 3 \ints)$ and a presentation of $\Gamma$ is given with generators $T$ and $U$ subject to 
\begin{equation*}
(TU^{-1}T)^2 = (U^{-1}T)^3 = 1 
\end{equation*}
where we identify $T$ and $U$ respectively with the images of the matrices 
\begin{equation*}
\begin{pmatrix} 1 & 1 \\ 0 & 1 \end{pmatrix} \ \text{and} \  \begin{pmatrix} 1 & 0 \\ 1 & 1 \end{pmatrix}
\end{equation*}
in $\Gamma$. Another presentation often encountered in the literature is given with generators $S$ and $T$ with the relations
\begin{equation*}
S^2 = (ST)^3 = 1
\end{equation*}
and this relates to the first presentation via $S = TU^{\inv}T$. Both sets of generators will be used throughout this paper depending on our preferences. We can treat an element of $\Gamma$ as a matrix $A$ with the understanding that $A\sim-A$ in $\Gamma$. Let us also make special notation for the principal congruence subgroups of $\Gamma$. Denote by $\Gamma(N)$ the principal congruence subgroup of level $N$. The proof of our main result requires having a reasonably sized (normally) generating set for $\Gamma(N)$ and so we record this information here. The following is from \cite{Hsu}, where the proof can be found. In what follows, for a subset $K$ of a group $G$, we denote by $\langle \langle K \rangle \rangle $ the smallest normal subgroup of $G$ containing $K$. 

\begin{proposition}{\label{Hsu}} Let $N$ be an integer greater than one. Write $N=ek$ where $e$ is a power of two and $k$ is odd. 
\begin{enumerate}[(i)]

\item ($N$ is odd) Suppose $e=1$ and let $t(N)$ be the multiplicative inverse of 2 mod $N$. Let 

\begin{equation*}
G_N = \Set{ T^N, \, (U^2T^{-t(N)})^3}.
\end{equation*}

\item ($N$ is a power of two) Suppose $k=1$ and let $f(N)$ be the multiplicative inverse of 5 mod $N$. Set $P_N = T^{20} U^{f(N)} T^{-4} U^{-1}$ and let 
\begin{equation*}
G_N = \Set{ T^N, \,  (P_N U^5 TU^{\inv}T)^3, \, (TU^{\inv}T)^{-1}P_N (TU^{\inv}T) P_N }  
\end{equation*}

\item ($N$ even, not a power of two) Suppose $e>1$ and $k>1$. Let $c$ be the unique integer mod $N$ so that 
\begin{align*}
c &= 0 \mod e \\
c &= 1 \mod k 
\end{align*}
and let $d$ be the unique integer mod $N$ so that 
\begin{align*}
d &= 0 \mod k\\
d &=1 \mod e.
\end{align*}
Write $t(N)$ for the multiplicative inverse of 2 mod $k$ and $f(N)$ for  the multiplicative inverse of 5 mod $e$. Set

\begin{align*}
x &= T^c \quad & z =T^d \\
y &=U^c \quad & w  = U^d
\end{align*}
and $p_N = z^{20} w^{f(N)} z^{-4} w^{-1}$. Let $G_N$ be equal to 
\begin{multline*}
\left\{ T^N, \, [x,w], \, (xy^{\inv}x)^4, \, (xy^{\inv}x)^2(x^{\inv}y)^3, \, (xy^{\inv}x)^2(x^{t(N)}y^{-2})^3  \right. \\
\left. (zw^{\inv}z)^2(p_N w^5 zw^{\inv} z)^{-3}, \, (zw^{\inv}z)^{\inv} p_N (zw^{\inv}z)p_N, \, w^{25} p_N w^{\inv} p_N ^{\inv} \right\} . 
\end{multline*}
\end{enumerate}

Then in any of the above cases, $\Gamma(N) = \langle \langle G_N \rangle \rangle$. 

\end{proposition}

Suppose $G$ is a congruence subgroup of $\Gamma$. Since each of the sets in Proposition \ref{Hsu} contains  
\begin{equation*}
T^N = \begin{pmatrix} 1 & N \\ 0 & 1 \end{pmatrix} 
\end{equation*}
and since $\Gamma(N)$ is a normal subgroup, $G$ contains each of the conjugates of $T^N$. Of course, an arbitrary finite index subgroup contains $T^N$ for some integer $N$ as well. With this in mind, we make two definitions. 

\begin{definition} Let $G$ be a subgroup of $\Gamma$ with finite index $\mu$ and let $\phi: \Gamma \to S_\mu$ be the coset representation afforded by $G$. Define the {\bf{geometric level}} of $G$ to be $\glevel(G) = |\phi(T)|$. It is an exercise to show 
\begin{equation*}
\glevel(G) = \min\Set{ N \geq 1  |  \langle \langle T^N \rangle \rangle \subseteq G }.
\end{equation*} If we further take $G$ to be congruence subgroup of $\Gamma$ then we can define the {\bf{level}} of $G$ to be 
\begin{equation*} 
\level(G) = \min \Set{ N | \Gamma(N) \subseteq G }.
\end{equation*}
Of course, if $N = \level(G)$ then $G$ contains $\Gamma(M)$ for each multiple $M$ of $N$. 
\end{definition}

We have an easy consequence of the definition of geometric level in the case of a normal subgroup. 

\begin{proposition}{\label{keyprop}} Suppose $H$ is a finite group and $\varphi: \Gamma \to H$ is a homomorphism. Then $\glevel(\ker \varphi) = |\varphi(T)|$. 
\end{proposition}

\begin{proof} Since $ \ker \varphi$ is a normal subgroup of $\Gamma$, we see that 
\begin{equation*}
\glevel(\ker \varphi) = \min\Set{N \geq 1  | T^N \in \ker \varphi } = |\varphi(T)|
\end{equation*} as desired. \end{proof}

\begin{observation}{\label{glevelless}}
For a congruence subgroup $G$, we see immediately from their definitions that $\glevel(G) \leq \level(G)$. In fact, it is easy to prove $\glevel(G)$ divides $\level(G)$, although we will not need this. 
\end{observation}

\subsection{The Fricke-Wohlfahrt theorem} Here we further explore the relationship between the two notions of level introduce above. Observation \ref{glevelless} leads us to investigate the reverse inequality. While both these numbers need not be defined for arbitrary finite index subgroups of $\Gamma$, it is the case that when both are defined,  they must agree. This result was originally proved by Fricke \cite{fricke} and later explored and expanded upon by Wohlfahrt \cite {Wohl} and therefore bears both of their names.  The majority of the proof is straightforward applications of linear congruences. This result is not entirely elementary, however, since it relies on the existence of arbitrary large primes in certain arithmetic progressions, a highly nontrivial fact from analytic number theory. 

\begin{theorem}[Fricke-Wohlfahrt theorem]{\label{fricke}} Suppose $G$ is a congruence subgroup of $\Gamma$. Then the level of $G$ equals the geometric level of $G$.  
\end{theorem}

As a corollary, we apply Proposition \ref{keyprop} to get the central tool in the proof of our main theorem. 

\begin{corollary}{\label{maincorollary}} Suppose $H$ is a finite group and $\varphi: \Gamma \to H$ is a homomorphism and let $m = |\varphi(T)|$. Then $\ker \varphi$ is a congruence subgroup if and only if $\Gamma(m) \subseteq \ker \varphi$. 
\end{corollary}

\begin{proof}[Proof of Fricke-Wohlfahrt]   Let $m = \glevel(G)$ and $n = \level(G)$. We will prove that $\Gamma(m) \subseteq G$. Since $G$ has level $n$, this implies $n \leq m$  and therefore $m = n$ by Observation \ref{glevelless}. More specifically, we will define subsets $\mathcal{E}_1, \mathcal{E}_2$, and $\mathcal{E}_3$ of $\Gamma(m)$ whose union equals $\Gamma(m)$ and show that each $\mathcal{E}_i$ is a subset of $G$. To this end, let $A \in \Gamma(m)$ and write 
\begin{equation*}
A = \begin{pmatrix} a & b \\ c & d \end{pmatrix}
\end{equation*} where (without loss of generality) $a=d=1$ mod $m$, $b=c=0$ mod $m$, and $ad-bc=1$. 

First, let $\mathcal{E}_1$ be the collection of elements whose upper right and lower left entries are divisible by $n$ and suppose $A \in \mathcal{E}_1$. If we let 
\begin{equation*}
A_0 = \begin{pmatrix} a & ad-1 \\ 1-ad & d(2-ad) \end{pmatrix} 
\end{equation*}
then since $ad-bc =1$ we have $ad=1 \mod n$ and so $A = A_0 \mod n$. Equivalently, $AA_0^{\inv} \in \Gamma(n) \subseteq G$. It suffices to show $A_0 \in G$. This follows from writing $V = TU^{\inv}T^3U^{\inv}T$ and computing $A_0 = (ST^{d-1}S)^{\inv} (VT^{a-1}V^{\inv} )T^{d-1}$  is an element of $\normcl{T^m}$ since $a=d=1 \mod m$.

Now let $\mathcal{E}_2$ be the set of matrices who diagonal entries are relatively prime to $n$. If $A \in \mathcal{E}_2$ then $\gcd(a,n)=1$ so $a$ is a unit mod $n$ say with inverse $a^\prime$. We can write $c = m\tilde{c}$ and let $ k = -a^\prime \tilde{c}$. Then 
\begin{equation*}
(ST^mS)^{-k} A = \begin{pmatrix} 1 & 0 \\ mk & 1 \end{pmatrix} \begin{pmatrix} a & b \\ c & d \end{pmatrix} = \begin{pmatrix} a & b \\ c+amk & d+bmk\end{pmatrix}
\end{equation*}
and $c+amk = m(\tilde{c} -a^\prime a \tilde{c})$ is congruent to 0 mod $n$. Let $\tilde{d} = d+bmk$. Taking the determinant of $(ST^mS)^{-k} A $, we see that $\tilde{d}$ is relatively prime to $n$. Therefore we can also choose $\ell$ so that $b+\tilde{d}m\ell =0$ mod $n$. In this case, 
\begin{equation*} T^{m\ell} (ST^mS)^{-k} A  = \begin{pmatrix} 1 & m\ell \\ 0 & 1 \end{pmatrix} \begin{pmatrix} a & b \\ c+amk & \tilde{d} \end{pmatrix} = \begin{pmatrix} \ast & b+\tilde{d}m\ell \\ c+amk &  \tilde{d} \end{pmatrix}
\end{equation*}
and this is an element of $\mathcal{E}_1$. This says $A$ is an element of $\normcl{T^m}\mathcal{E}_1$ and is therefore in $G$. 

Finally, define $\mathcal{E}_3  =\Gamma(m) \setminus (\mathcal{E}_1 \cup \mathcal{E}_2)$. If $A \in \mathcal{E}_3$ then $\gcd(a,n)\not =1$ or $\gcd(d,n) \not= 1$. Suppose first that $\gcd(a,n)\not =1$. Since $ad-bc=1$ we know $\gcd(a,b)= 1$. Then since $a =1$ mod $m$, we also have $\gcd(a,bm) = 1$. Recall Dirichlet's theorem on primes in arithmetic progressions: Suppose $r$ and $s$ are relatively prime integers. Then 
\begin{equation*}
\Set{ r+st | t \in \ints }
\end{equation*}
contains infinitely many prime numbers. In particular, taking $r=a$ and $s=bm$ we can choose an integer $k$ so that $a+bmk$ is a prime number larger than $n$. Then 
\begin{equation*}
A(ST^mS)^{-k}=\begin{pmatrix} a & b \\ c & d \end{pmatrix} \begin{pmatrix} 1 & 0 \\ mk & 1 \end{pmatrix} = \begin{pmatrix} a + bmk & b \\ c+ dmk & d  \end{pmatrix} 
\end{equation*}
has upper left entry relatively prime to $n$. If $\gcd(d,n) = 1$ then $A(ST^mS)^{-k} \in \mathcal{E}_2$ and we are done. Otherwise, let $\tilde{a} = a+bmk$, let $\tilde{c} = c+dmk$ and note that $\tilde{a}d-b\tilde{c} =1$ so $\gcd(\tilde{c},d) = 1$. As $\gcd(d,m)=1$, just as above we can choose $\ell$ so that $d+\tilde{c}m\ell$ is a prime number larger than $n$. Then 
\begin{equation*}
A(ST^mS)^{-k} T^{m\ell} = \begin{pmatrix} \tilde{a} & b \\ \tilde{c} & d \end{pmatrix} \begin{pmatrix} 1 & m\ell \\ 0 & 1 \end{pmatrix} = \begin{pmatrix} \tilde{a} & \ast \\ \tilde{c} & d+\tilde{c}mk \end{pmatrix}
\end{equation*}
and $A(ST^mS)^{-k} T^{m\ell}$ must be in $\mathcal{E}_2$ Thus $A$ is an element of  $\mathcal{E}_2 \normcl{T^m}$ and the theorem follows. 
\end{proof}

\subsection{Low-dimensional representations of $B_3$} In this section collect some known results about representations of $B_3$ of small dimension. One major theorem in this direction is the Tuba-Wenzl classification. Indeed, a representation of dimension between two and five can be conjugated so that $\sigma_1$ (resp. $\sigma_2$) is mapped to an upper (resp. lower) triangular matrix. We record the two and three-dimensional version here. 

\begin{theorem}[TW classificiation in dimension two and three]{\label{TW}}\leavevmode
\begin{enumerate}[(i)] \item Let $N_2$ be the zero set of $\lambda_1^1 + \lambda_1\lambda_2  +\lambda_2^2$ and let $N_3$ be the zero set of $\Set{ \lambda_j^2 + \lambda_k\lambda_{\ell} | \set{j,k,l} = \set{1,2,3}}$. Then for $d=2$ or $3$ there is a bijection between conjugacy classes of irreducible $d$-dimensional representations of $B_3$ and $S_d$-orbits of $\C^d \setminus N_d$. 

\item Suppose $\rho: B_3 \to \mathrm{GL}(2,\C)$ is irreducible and $\spec(\rho(\sigma_1)) = \set{\lambda_1, \lambda_2}$. Then up to conjugation, 
\begin{equation*}
\rho(\sigma_1) = \begin{pmatrix} \lambda_1 & \lambda_1 \\ 0 & \lambda_2 \end{pmatrix} \quad \rho(\sigma_2) = \begin{pmatrix} \lambda_2 & 0 \\ - \lambda_2 & \lambda_1 \end{pmatrix}.
\end{equation*}
Furthermore, $\rho$ factors through $\pi$ if and only if $-(\lambda_1\lambda_2)^3=1$. 

\item Suppose $\rho: B_3 \to \mathrm{GL}(3,\C)$ is irreducible and $\spec(\rho(\sigma_1)) = \set{\lambda_1, \lambda_2, \lambda_3}$. Then up to conjugation, 
\begin{equation*}
\rho(\sigma_1) = \begin{pmatrix} \lambda_1 & \lambda_1\lambda_3\lambda_2^{\inv} + \lambda_2 & \lambda_2 \\ 0 & \lambda_2 & \lambda_2 \\ 0 & 0 & \lambda_3 \end{pmatrix} \quad \rho(\sigma_2) = \begin{pmatrix} \lambda_3 & 0 & 0 \\ -\lambda_2 & \lambda_2 & 0 \\ \lambda_2 & -\lambda_1\lambda_3\lambda_2^{\inv} - \lambda_2 & \lambda_1 \end{pmatrix} . 
\end{equation*}
Furthermore, $\rho$ factors through $\pi$ if and only if $(\lambda_1\lambda_2\lambda_3)^2 = 1$. 
\end{enumerate}
\end{theorem}

The task of determining if the image of an irreducible representation of dimension between two and five is finite or infinite was initiated in \cite{RT} and there the authors' main result showed that for a two or three-dimensional representation this is a property of the eigenvalues of the image of $\sigma_1$ under the representation. A version of their results in dimensions two and three is included below.

\begin{definition} Let $A$ be an $n \times n$ matrix with eigenvalues $\lambda_1, \lambda_2, \ldots, \lambda_n$. Define the {\bf{projective order}} of $A$ to be 
\begin{equation*}
\po(A) = \min\Set{  t>1  |  \lambda_1^t  = \cdots = \lambda_n^t }
\end{equation*}
where this is allowed to be infinite. The projective order of a matrix is invariant under scaling. That is, $\po(A) = \po(\theta A)$ for all $\theta \in \C^\ast$. 
\end{definition}

\begin{theorem}{\label{RT}}Let $d=2$ or $3$ and suppose $\rho: B_3 \to \mathrm{GL}(d,\C)$ is irreducible with $\spec(\rho(\sigma_1)) = \{\lambda_1, \ldots, \lambda_d\}$. 
\begin{enumerate}[(a)]
\item If some $\lambda_i$ is not a root of unity then the image of $\rho$ is infinite. 
\item If $\lambda_i = \lambda_j$ for some $i \not= j$ then the image of $\rho$ is infinite. 
\item If $\set{ \lambda_1, \ldots, \lambda_d}$ consists of distinct roots of unity and $2 \leq \po(\rho(\sigma_1)) \leq 5$ then the image of $\rho$ is finite. 
\item If $d=2$ and the image of $\rho$ is finite then $2 \leq \po(\rho(\sigma_1)) \leq 5$. 
\item If $d=3$ and $\set{\lambda_1, \, \lambda_2, \, \lambda_3} = \set{\pm \lambda, \mu}$ for some distinct roots of unity $\lambda,\mu$ then the image of $\rho$ is finite. 
\end{enumerate}

\end{theorem}

\section{Main results}

This section is devoted to proving our main results. Recall that $Z(B_3)$ is the cyclic group generated by $(\sigma_1\sigma_2)^3$ and we have a surjection $\pi: B_3 \to \Gamma$ whose kernel is $Z(B_3)$. Explicitly, we can take $\pi(\sigma_1) = T$ and $\pi(\sigma_2) = U^{-1}$. We first establish results for $\Gamma$ that handles most of the work for our main theorem. Let us begin in dimension two.

\begin{proposition}{\label{maingamma2}}Suppose $\rho: \Gamma \to \mathrm{GL}(2, \C)$ is irreducible and has finite image and let $N = |\rho(T)|$. Then $\ker \rho$ is a congruence subgroup of level $N$. 
\end{proposition} 

\begin{proof} It is clear from the definition that the geometric level of $\ker \rho$ is $N$. Therefore, by \ref{maincorollary}, we are done if we can show that $\ker \rho$ is a congruence subgroup. To this end, we will show that $\ker \rho$ contains the generating set of $G_N$ given in Proposition \ref{Hsu}. 
Denote by $\lambda_1$ and $\lambda_2$ the eigenvalues of $\rho(T)$. Since the image of $\rho$ is finite, the projective order $r$ of $\rho(T)$ must be finite. Then Theorem \ref{RT} implies that we can write $\lambda_2 = e^{2\pi i j /r} \lambda_1$ where  $ 2\leq r \leq 5$ and $ j \in \ints_r^\times$. Define $\tilde{\rho} = \rho \circ \pi$. This is a representation of $B_3$ that factors through $\pi$ and thus $(\sigma_1\sigma_2)^3$ is in the kernel of this map.  We computed earlier that $(\sigma_1 \sigma_2)^3$ acts by $-(\lambda_1\lambda_2)^3$ under $\rho$ and hence $\lambda_1$ satisfies
\begin{equation}{\label{key2}}
\lambda_1^6 + e^{-6 \pi i j/r} = 0
\end{equation} 
since $(\sigma_1\sigma_2)^3$ is in the kernel of $\tilde{\rho}$. This polynomial equation determines the possibilities for $\lambda_1$ and hence for $\rho$ by Theorem \ref{TW}.  Each allowable choice of $r$ and $j$ provides six representations $\rho_{r,j,\lambda}$ corresponding to the six solutions $\lambda$ to (\ref{key2}) and $\rho$ must be equivalent to $\rho_{r,j,\lambda}$ for some $r$ and $j$ and solution $\lambda$ to (\ref{key2}). To proceed, we show the result holds for each $\rho_{r,j,\lambda}$. For the rest of the proof let us write $X = \rho(T)$ and $Y = \rho(U)$. The key observation we use repeatedly is that $X^r = \lambda^r I$ since $\po(X) = r$ and $X$ has $\lambda$ as an eigenvalue. Also $Y^r = \overline{\lambda}^r I$ and by \cite{RT} there is a symmetry $\rho_{r,j, \lambda} = \rho_{r,-j, e^{2\pi i j/r} \lambda}$ which reduces the number of cases to consider.

{\bf{{\underline{$r=2$:}}}} The only option for $\xi_2$ is -1, so $\spec(\rho_{2,1,\lambda}(T)) = \Set{\lambda, -\lambda}$ where $\lambda^6  = 1 $. When $\lambda = 1$ or $-1$ then $N=2$ and $X^2 = I$. Here $P_2 = T^{20} U T^{-4} U^{-1}$ and so $\rho_{2,1,\lambda}(P_2) = X^{20} Y X^{-4} Y^{-1} = Y Y^{-1} = I$. In turn, 
\begin{align*}
\rho_{2,1,\lambda}(T^2) &= X^2 = I\\
\rho_{2,1,\lambda}(P_2U^5TU^{-1}T)^3 &= (Y^5 X Y^{-1} X)^3 = (Y^{-1}X)^6 = (\rho_{2,1,\lambda}(U^{-1}T)^3)^2 = I 
\end{align*}
which shows $\Gamma(2) \subseteq \ker \rho_{2,1,\lambda}$. 

Now, $\lambda= e^{2 \pi i /6}$ or $e^{4 \pi i /3}$ implies $N=6$ and $X^2 = e^{2\pi i /3}I$. Here $c=4$ and $d=3$ so that 
\begin{align*}
\rho_{2,1,\lambda}(x)  &=  X^4 = e^{4 \pi i 3} I   \quad   & \rho_{2,1,\lambda}(z)  = X^3 = e^{2 \pi i /3} X \\
\rho_{2,1,\lambda}(y) &= Y^4 = e^{2 \pi i /3}I  \quad    & \rho_{2,1,\lambda}(w)  = Y^3 = e^{4 \pi i /3} Y 
\end{align*}
and $\rho_{2,1,\lambda}(p_6) = \rho_{2,1,\lambda}(z^{20} w z^{-4}w^{-1}) = I$. We see $T^6 \in \ker \rho_{2,1,\lambda}$ and since $\rho_{2,1,\lambda}(x)$ is a scalar matrix, $[x,w]$ is also in the kernel of $\rho_{2,1,\lambda}$. We have
\begin{equation*}
\rho_{2,1,\lambda}(xy^{-1}x)^4 = (e^{4 \pi i /3} I) ^{12} = I
\end{equation*}
so that $(xy^{\inv}x)^4 \in \ker \rho_{2,1,\lambda}$.  Also
\begin{align*}
\rho_{2,1,\lambda}(xy^{\inv}x)^2 &= (e^{4 \pi i /3}I)^6 = I \\
\rho_{2,1,\lambda}(x^{\inv}y)^3 &= (e^{4 \pi i /3}I)^3 = I \\
\rho_{2,1,\lambda}(x^2 y^{-2})^3 &= (e^{4 \pi i/3}I)^3 = I.
\end{align*}
Further, 
\begin{align*}
\rho_{2,1,\lambda}(zw^{\inv}z)^2  &= ((e^{2 \pi i /3})^3 XY^{\inv} X)^2 = \rho((TU^{\inv}T)^2) = I \\
\rho_{2,1,\lambda}(p_6 w^5 zw^{\inv}z)^3  & = (XX^{\inv}YXY^{\inv}X)^3 =  (XY^{\inv})^3 = I
\end{align*}
since $U^{\inv} = T^{\inv} U T U^{\inv} X$. Lastly, $\rho_{2,1,\lambda}(w^{25}) = \rho_{2,1,\lambda}(w)$, and so we see that $\Gamma(6) \subseteq \ker \rho_{2,1,\lambda}$. 

If $\lambda = e^{2\pi i /3}$ or $e^{10 \pi i /6}$ then $N=6$ and $X^2 = e^{4 \pi i /3} I$. Again $c=4$ and $d=3$ so that 
\begin{align*}
\rho_{2,1,\lambda}(x)  &=  X^4 = e^{2 \pi i 3} I  \quad  & \rho_{2,1,\lambda}(z)  = X^3 = e^{4 \pi i /3} X \\
\rho_{2,1,\lambda}(y) & = Y^4 = e^{4 \pi i /3}I       \quad &    \rho_{2,1,\lambda}(w)  = Y^3 = e^{2 \pi i /3} Y 
\end{align*}
and $\rho_{2,1,\lambda}(p_6) = \rho_{2,1,\lambda}(z^{20}  w z^{-4} w^{-1}) = I$. We see $T^6 \in \ker \rho_{2,1,\lambda}$ and since $\rho_{2,1,\lambda}(x)$ is a scalar matrix, $[x,w]$ is also in the kernel of $\rho_{2,1,\lambda}$. We have
\begin{equation*}
\rho_{2,1,\lambda}(xy^{-1}x)^4 = (e^{2 \pi i /3} I) ^{12} = I
\end{equation*}
so that $(xy^{\inv}x)^4 \in \ker \rho_{2,1,\lambda}$.  Also
\begin{align*}
\rho_{2,1,\lambda}(xy^{\inv}x)^2 &= (e^{2 \pi i /3}I)^6 = I \\
\rho_{2,1,\lambda}(x^{\inv}y)^3 &= (e^{2 \pi i /3}I)^3 = I \\
\rho_{2,1,\lambda}(x^2 y^{-2})^3 &= (e^{2 \pi i/3}I)^3 = I.
\end{align*}
Further, 
\begin{align*}
\rho_{2,1,\lambda}(zw^{\inv}z)^2  &= ((e^{4 \pi i /3})^3 XY^{\inv} X)^2 = \rho((TU^{\inv}T)^2) = I \\
\rho_{2,1,\lambda}(p_6 w^5 zw^{\inv}z)^3  & = (XX^{\inv}YXY^{\inv}X)^3 =  (XY^{\inv})^3 = I
\end{align*}
since $U^{\inv} = T^{\inv} U T U^{\inv} X$. Lastly, $\rho_{2,1,\lambda}(w^{25}) = \rho_{2,1,\lambda}(w)$, and so we see that $\Gamma(6) \subseteq \ker \rho_{2,1,\lambda}$.

{\bf{{\underline{$r=3$}:}}} We can have either $\xi_3 = e^{2\pi i /3}$ or $e^{4 \pi i /3}$ but by the symmetry mentioned above we can take $\xi_3 =e^{2 \pi i /3}$. If $\lambda = e^{\pi i /6}, \, e^{5 \pi i /6}, \,$ or $e^{ 9 \pi i/6}$ then $N=12$ and $X^3 =iI$. Here $c=4$ and $d=9$ so that 
\begin{align*}
\rho_{3,1,\lambda}(x)  &=  X^4 = iX   \quad &\rho_{3,1,\lambda}(z)  = X^9 = -iI \\
\rho_{3,1,\lambda}(y) &= Y^4 = -iX  \quad  &\rho_{3,1,\lambda}(w) = Y^9 = iI 
\end{align*}
and $\rho_{3,1,\lambda}(p_{12}) = \rho_{3,1,\lambda}(z^{20} w z^{-4} w^{-1}) = I$. We see $T^{12} \in \ker \rho_{3,1,\lambda}$. Since $\rho_{3,1,\lambda}(w)$ is a scalar matrix, $[x,w]$ is also in the kernel of $\rho_{3,1,\lambda}$. We have
\begin{equation*}
\rho_{3,1,\lambda}(xy^{-1}x)^4 = (-iXY^{\inv}X)^4 = I
\end{equation*}
so that $(xy^{\inv}x)^4 \in \ker \rho_{3,1,\lambda}$.  Also
\begin{align*}
\rho_{3,1,\lambda}(xy^{\inv}x)^2 &= (-iXY^{\inv}X)^2 = -I \\
\rho_{3,1,\lambda}(x^{\inv}y)^3 &= (-X^{\inv}Y)^3 = -I \\
\rho_{3,1,\lambda}(x^2 y^{-2})^3 &= (X^2Y^{-2})^3 = (X^2Y^{10})^3 = -i[(Y^{\inv} X^{10})^3]^{\inv} = - [(Y^{\inv}X)^3]^{\inv} -I. 
\end{align*}
Further, 
\begin{align*}
\rho_{3,1,\lambda}(zw^{\inv}z)^2  &= - I \\
\rho_{3,1,\lambda}(p_{12} w^5 zw^{\inv}z)^3  &= \rho_{3,1,\lambda}(w  z w^{\inv} z) = -I.
\end{align*}
Lastly, $\rho_{3,1,\lambda}(w)^{25} = \rho_{3,1,\lambda}(w)$, and so we see that $\Gamma(12) \subseteq \ker \rho_{3,1,\lambda}$.

If $\lambda_1 = e^{3 \pi i /6}, \, e^{7 \pi i /6}, \,$ or $e^{11 \pi i /6}$ then $N=12$ and $X^3 =-iI$. Here $c=4$ and $d=9$ so that 
\begin{align*}
\rho_{3,1,\lambda}(x)  &=  X^4 = -iX  \quad & \rho_{3,1,\lambda}(z) & = X^9 = iI \\
\rho_{3,1,\lambda}(y) &= Y^4 = iX \quad & \rho_{3,1,\lambda}(w) & = Y^9 = -iI 
\end{align*}
and $\rho_{3,1,\lambda}(p_{12}) = \rho_{3,1,\lambda}(z^{20}  w z^{-4}  w^{-1}) = I$. We see $T^{12} \in \ker \rho_{3,1,\lambda}$. Since $\rho_{3,1,\lambda}(w)$ is a scalar matrix, $[x,w]$ is also in the kernel of $\rho_{3,1,\lambda}$. We have
\begin{equation*}
\rho_{3,1,\lambda}(xy^{-1}x)^4 = (iXY^{\inv}X)^4 = I
\end{equation*}
so that $(xy^{\inv}x)^4 \in \ker \rho_{3,1,\lambda}$.  Also
\begin{align*}
\rho_{3,1,\lambda}(xy^{\inv}x)^2 &= (iXY^{\inv}X)^2 = -I \\
\rho_{3,1,\lambda}(x^{\inv}y)^3 &= (-X^{\inv}Y)^3 = -I \\
\rho_{3,1,\lambda}(x^2 y^{-2})^3 &= (X^2Y^{-2})^3 = (X^2Y^{10})^3 = i[(Y^{\inv} X^{10})^3]^{\inv} = - [(Y^{\inv}X)^3]^{\inv} -I. 
\end{align*}
Further, 
\begin{align*}
\rho_{3,1,\lambda}(zw^{\inv}z)^2  &= - I \\
\rho_{3,1,\lambda}(p_{12} w^5 zw^{\inv}z)^3  &= \rho_{3,1,\lambda}(w z w^{\inv} z) = -I.
\end{align*}
Lastly, $\rho_{3,1,\lambda}(w)^{25} = \rho_{3,1,\lambda}(w)$, and so we see that $\Gamma(12) \subseteq \ker \rho_{3,1,\lambda}$.

{\bf{{\underline{$r=4$}:}}} We can take $\xi_4 = i$ so that $\lambda_1^6 = -i$. When $\lambda = e^{3 \pi i /12}$ or $e^{15 \pi i /12}$ we have $N=8$ and $X^4 = -I$. Here $P_8 = T^{20} U^5 T^{-4} U^{-1}$ and so $\rho_{4,1,\lambda}(P_8) = X^{20} Y^{5} X^{-4} Y^{-1} =X^4 Y^5 X^{-4} Y^{-1} =  Y^4 = -I$. In turn, 
\begin{align*}
\rho_{4,1,\lambda}(T^2) &= X^2 = I\\
\rho_{4,1,\lambda}(P_8U^5TU^{-1}T)^3 &= ( -Y^5 X Y^{\inv} X)^3 = (YXY^{\inv}X)^3 = (XY^{\inv})^3 = I
\end{align*}
which shows $\Gamma(8) \subseteq \ker \rho_{4,1,\lambda}$. If $\lambda_1 = e^{7 \pi i /12}$ or $e^{19 \pi i /12}$ then $N=24$ and $X^4 = e^{ \pi i /3 }I$. Here $c=16$ and $d=9$ so that 
\begin{align*}
\rho_{4,1,\lambda}(x)  &=  X^16 = e^{4\pi i /3}I  \quad & \rho_{4,1,\lambda}(z)  = X^9 = e^{2 \pi i /3} X \\
\rho_{4,1,\lambda}(y) &= Y^16 = e^{2 \pi i 3} I  \quad & \rho_{4,1,\lambda}(w) & = Y^9 = e^{4 \pi i 3} Y
\end{align*}
and 
\begin{equation*}
\rho_{4,1,\lambda}(p_{24}) = \rho_{4,1,\lambda}(z^{20}  w^5 z^{-4} w^{-1}) = \rho_{4,1,\lambda}(z^4 w^5 z^{-4} w^{-1}) = \rho(w)^4 = e^{4 \pi i /3} Y^4 = -I.
\end{equation*}

We see $T^{24} \in \ker \rho_{4,1,\lambda}$. Since $\rho_{4,1,\lambda}(x)$ is a scalar matrix, $[x,w]$ is also in the kernel of $\rho_{4,1,\lambda}$. We have
\begin{equation*}
\rho_{4,1,\lambda}(xy^{-1}x) = (e^{4 \pi i /3} I)^3 = I
\end{equation*}
so that $(xy^{\inv}x)^4 \in \ker \rho_{4,1,\lambda}$.  Also
\begin{align*}
\rho_{4,1,\lambda}(xy^{\inv}x)^2 &= I \\
\rho_{4,1,\lambda}(x^{\inv}y)^3 &= (e^{4\pi i /3}I)^3 = I \\
\rho_{4,1,\lambda}(x^2 y^{-2})^3 &= (e^{4 \pi i/3} I)^3 = I .
\end{align*}
Further, 
\begin{equation*}
\rho_{4,1,\lambda}(zw^{\inv}z) = XY^{\inv}X
\end{equation*}
so that
\begin{equation*}
\rho_{4,1,\lambda}(zw^{\inv}z)^2  = I 
\end{equation*} and
\begin{equation*} 
\rho_{4,1,\lambda}(p_{24} w^5 zw^{\inv}z)^3  =  (\rho_{4,1,\lambda}(w z w^{\inv} z))^3 = (e^{4 \pi i /3}YXY^{\inv}X)^3 = (XY^{\inv})^3 = I.
\end{equation*}
Lastly, $\rho_{4,1,\lambda}(w)^{25} = \rho_{4,1,\lambda}(w)$, and so we see that $\Gamma(24) \subseteq \ker \rho_{4,1,\lambda}$. If $\lambda_1 = e^{11 \pi i /12}$ or $e^{23 \pi i /12}$ then $N=24$ and $X^4 = e^{ 5 \pi i /3 }I$. Again $c=16$ and $d=9$ so that 
\begin{align*}
\rho_{4,1,\lambda}(x)  &=  X^16 = e^{2\pi i /3}I  \quad & \rho_{4,1,\lambda}(z) = X^9 = e^{4 \pi i /3} X \\
\rho_{4,1,\lambda}(y) &= Y^16 = e^{4 \pi i 3} I  \quad & \rho_{4,1,\lambda}(w) = Y^9 = e^{2 \pi i 3} Y
\end{align*}
and 
\begin{equation*}
\rho_{4,1,\lambda}(p_{24}) = \rho_{4,1,\lambda}(z^{20} w^5 z^{-4} w^{-1}) = \rho_{4,1,\lambda}(z^4 w^5 z^{-4}w^{-1}) = \rho_{4,1,\lambda}(w)^4 = e^{2 \pi i /3} Y^4 = -I.
\end{equation*}

We see $T^{24} \in \ker \rho_{4,1,\lambda}$. Since $\rho_{4,1,\lambda}(x)$ is a scalar matrix, $[x,w]$ is also in the kernel of $\rho_{4,1,\lambda}$. We have
\begin{equation*}
\rho_{4,1,\lambda}(xy^{-1}x) = (e^{2 \pi i /3} I)^3 = I
\end{equation*}
so that $(xy^{\inv}x)^4 \in \ker \rho_{4,1,\lambda}$.  Also
\begin{align*}
\rho_{4,1,\lambda}(xy^{\inv}x)^2 &= I \\
\rho_{4,1,\lambda}(x^{\inv}y)^3 &= (e^{2\pi i /3}I)^3 = I \\
\rho_{4,1,\lambda}(x^2 y^{-2})^3 &= (e^{2 \pi i/3} I)^3 = I .
\end{align*}
Further, 
\begin{equation*}
\rho_{4,1,\lambda}(zw^{\inv}z) = XY^{\inv}X
\end{equation*}
so that
\begin{equation*}
\rho_{4,1,\lambda}(zw^{\inv}z)^2  = I 
\end{equation*} and
\begin{equation*} 
\rho_{4,1,\lambda}(p_{24} w^5 zw^{\inv}z)^3  =  (\rho_{4,1,\lambda}(w zw^{\inv} z))^3 = (e^{2 \pi i /3}YXY^{\inv}X)^3 = (XY^{\inv})^3 = I.
\end{equation*}
Lastly, $\rho_{4,1,\lambda}(w)^{25} = \rho_{4,1,\lambda}(w)$, and so we see that $\Gamma(24) \subseteq \ker \rho$.

{\bf{{\underline{$r=5$}:}}}  There are two cases to consider for $r=5$. We can have either $\xi_5=e^{2 \pi i /5}$ or $e^{4\pi i /5}$. Let us take $\xi_5 = e^{2 \pi i/5}$. Then $\lambda^6 = -e^{4 \pi i /5}$. If $\lambda = e^{3 \pi i /10}$ then $N=20$ and $X^5 = - i I$. Here $c=16$ and $d=5$ so that 
\begin{align*}
\rho_{5,1,\lambda}(x)  &=  X^{16} = iX  \quad &\rho_{5,1,\lambda}(z)  = X^5 = -iI \\
\rho_{5,1,\lambda}(y) &= Y^{16} = -iY \quad & \rho_{5,1,\lambda}(w)  = Y^5 = iI
\end{align*}
and $\rho_{5,1,\lambda}(p_{20}) = \rho_{5,1,\lambda}(z^{20} w z^{-4} w^{-1}) = I$. We see $T^{20} \in \ker \rho_{5,1,\lambda}$. Since $\rho_{5,1,\lambda}(w)$ is a scalar matrix, $[x,w]$ is also in the kernel of $\rho_{5,1,\lambda}$. We have
\begin{equation*}
\rho_{5,1,\lambda}(xy^{-1}x)^4 = (-iXY^{\inv}X)^4 = I
\end{equation*}
so that $(xy^{\inv}x)^4 \in \ker \rho_{5,1,\lambda}$.  Also
\begin{align*}
\rho_{5,1,\lambda}(xy^{\inv}x)^2 &= (-iXY^{\inv}X)^2 = -I \\
\rho_{5,1,\lambda}(x^{\inv}y)^3 &= (-X^{\inv}Y)^3 = -I \\
\rho_{5,1,\lambda}(x^3 y^{-2})^3 &=  (iX^3Y^{-2})^3 = (iX^3 Y^{18})^3 = (X^3Y^3)^3 = -I. 
\end{align*}
Further, 
\begin{align*}
\rho_{5,1,\lambda}(zw^{\inv}z)^2  &=  -I \\
\rho_{5,1,\lambda}(p_{20} w^5 zw^{\inv}z)^3  &= (\rho_{5,1,\lambda}(w z w^{\inv} z))^3 = \rho_{5,1,\lambda}(z)^6 = -I
\end{align*}
Lastly, $\rho_{5,1,\lambda}(w)^{25} = \rho_{5,1,\lambda}(w)$, and so we see that $\Gamma(20) \subseteq \ker \rho_{5,1,\lambda}$.

If $\lambda = e^{19 \pi i /30}$ then $N=60$ and $X^5 = e^{7\pi i /6}  I$. Here $c=16$ and $d=5$ so that 
\begin{align*}
\rho_{5,1,\lambda}(x)  &=  X^{16} = -iX  \quad & \rho_{5,1,\lambda}(z)  = X^5 = e^{7 \pi i /6}I \\
\rho_{5,1,\lambda}(y) &= Y^{16} = iY \quad & \rho_{5,1,\lambda}(w)  = Y^5 = e^{5\pi i /6}I
\end{align*}
and $\rho_{5,1,\lambda}(p_{60}) = \rho_{5,1,\lambda}(z^{20} w z^{-4}  w^{-1}) = \rho_{5,1,\lambda}(z)^4 = e^{2 \pi i /3} I$. We see $T^{60} \in \ker \rho_{5,1,\lambda}$. Since $\rho_{5,1,\lambda}(w)$ is a scalar matrix, $[x,w]$ is also in the kernel of $\rho_{5,1,\lambda}$. We have
\begin{equation*}
\rho_{5,1,\lambda}(xy^{-1}x)^4 = (iXY^{\inv}X)^4 = I
\end{equation*}
so that $(xy^{\inv}x)^4 \in \ker \rho_{5,1,\lambda}$.  Also
\begin{align*}
\rho_{5,1,\lambda}(xy^{\inv}x)^2 &= (iXY^{\inv}X)^2 = -I \\
\rho_{5,1,\lambda}(x^{\inv}y)^3 &= (-X^{\inv}Y)^3 = -I \\
\rho_{5,1,\lambda}(x^8 y^{-2})^3 &=  (-X^8Y^{-2})^3 = (e^{2\pi i /3}X^3 Y^3)^3 = (X^3Y^3)^3 = -I
\end{align*}
Further, 
\begin{align*}
\rho_{5,1,\lambda}(zw^{\inv}z)^2  &=  -I \\
\rho_{5,1,\lambda}(p_{60} w^5 zw^{\inv}z)^3  &= (e^{2 \pi i /3} \rho_{5,1,\lambda}(w^5  z w^{\inv}  z))^3 = \rho(w^{12} z^6) = -I
\end{align*}
Lastly, $\rho_{5,1,\lambda}(w)^{25} = \rho_{5,1,\lambda}(w)$, and so we see that $\Gamma(60) \subseteq \ker \rho_{5,1,\lambda}$. 

If $\lambda= e^{29 \pi i /30}$ then $N=60$ and $X^5 = e^{5\pi i /6}  I$. Here $c=16$ and $d=5$ so that 
\begin{align*}
\rho_{5,1,\lambda}(x)  &=  X^{16} = iX  \quad & \rho_{5,1,\lambda}(z)  = X^5 = e^{5 \pi i /6}I \\
\rho_{5,1,\lambda}(y) &= Y^{16} = -iY \quad & \rho_{5,1,\lambda}(w)  = Y^5 = e^{7\pi i /6}I
\end{align*}
and $\rho_{5,1,\lambda}(p_{60}) = \rho_{5,1,\lambda}(z^{20} w z^{-4} w^{-1}) = \rho_{5,1,\lambda}(z)^4 = e^{4 \pi i /3} I$. We see $T^{60} \in \ker \rho_{5,1,\lambda}$. Since $\rho_{5,1,\lambda}(w)$ is a scalar matrix, $[x,w]$ is also in the kernel of $\rho_{5,1,\lambda}$. We have
\begin{equation*}
\rho_{5,1,\lambda}(xy^{-1}x)^4 = (-iXY^{\inv}X)^4 = I
\end{equation*}
so that $(xy^{\inv}x)^4 \in \ker \rho_{5,1,\lambda}$.  Also
\begin{align*}
\rho_{5,1,\lambda}(xy^{\inv}x)^2 &= (-iXY^{\inv}X)^2 = -I \\
\rho_{5,1,\lambda}(x^{\inv}y)^3 &= (-X^{\inv}Y)^3 = -I \\
\rho_{5,1,\lambda}(x^8 y^{-2})^3 &=  (-X^8Y^{-2})^3 = (e^{2\pi i /3}X^3 Y^3)^3 = (X^3Y^3)^3 = -I
\end{align*}
Further, 
\begin{align*}
\rho_{5,1,\lambda}(zw^{\inv}z)^2  &=  -I \\
\rho_{5,1,\lambda}(p_{60} w^5 zw^{\inv}z)^3  &= (e^{4 \pi i /3} \rho_{5,1,\lambda}(w z w^{\inv} z))^3 = \rho_{5,1,\lambda}(w^{12}z^6) =  -I
\end{align*}
Lastly, $\rho_{5,1,\lambda}(w)^{25} = \rho_{5,1,\lambda}(w)$, and so we see that $\Gamma(60) \subseteq \ker \rho_{5,1,\lambda}$.

If $\lambda =  e^{13 \pi i /10}$ then $N=20$ and $X^5 =  i I$. Here $c=16$ and $d=5$ so that 
\begin{align*}
\rho_{5,1,\lambda}(x)  &=  X^{16} = -iX  \quad & \rho_{5,1,\lambda}(z) = X^5 = iI \\
\rho_{5,1,\lambda}(y) &= Y^{16} = iY \quad & \rho_{5,1,\lambda}(w)  = Y^5 = -iI
\end{align*}
and $\rho_{5,1,\lambda}(p_{20}) = \rho_{5,1,\lambda}(z^{20}  w z^{-4}  w^{-1}) = I$. We see $T^{20} \in \ker \rho_{5,1,\lambda}$. Since $\rho_{5,1,\lambda}(w)$ is a scalar matrix, $[x,w]$ is also in the kernel of $\rho_{5,1,\lambda}$. We have
\begin{equation*}
\rho_{5,1,\lambda}(xy^{-1}x)^4 = (iXY^{\inv}X)^4 = I
\end{equation*}
so that $(xy^{\inv}x)^4 \in \ker \rho_{5,1,\lambda}$.  Also
\begin{align*}
\rho_{5,1,\lambda}(xy^{\inv}x)^2 &= (iXY^{\inv}X)^2 = -I \\
\rho_{5,1,\lambda}(x^{\inv}y)^3 &= (-X^{\inv}Y)^3 = -I \\
\rho_{5,1,\lambda}(x^3 y^{-2})^3 &=  (-iX^3Y^{-2})^3 = (-iX^3 Y^{18})^3 = (X^3Y^3)^3 = -I. 
\end{align*}
Further, 
\begin{align*}
\rho_{5,1,\lambda}(zw^{\inv}z)^2  &=  -I \\
\rho_{5,1,\lambda}(p_{20} w^5 zw^{\inv}z)^3  &= (\rho_{5,1,\lambda}(w z w^{\inv} z))^3 = \rho_{5,1,\lambda}(z)^6 = -I
\end{align*}
Lastly, $\rho_{5,1,\lambda}(w)^{25} = \rho_{5,1,\lambda}(w)$, and so we see that $\Gamma(20) \subseteq \ker \rho_{5,1,\lambda}$.

If $\lambda= e^{49 \pi i /30}$ then $N=60$ and $X^5 = e^{\pi i /6}  I$. Here $c=16$ and $d=5$ so that 
\begin{align*}
\rho_{5,1,\lambda}(x)  &=  X^{16} = iX  \quad &  \rho_{5,1,\lambda}(z)  = X^5 = e^{ \pi i /6}I \\
\rho_{5,1,\lambda}(y) &= Y^{16} = -iY \quad & \rho_{5,1,\lambda}(w)  = Y^5 = e^{11\pi i /6}I
\end{align*}
and $\rho_{5,1,\lambda}(p_{60}) = \rho_{5,1,\lambda}(z^{20} w z^{-4} w^{-1}) = \rho_{5,1,\lambda}(z)^4 = e^{2 \pi i /3} I$. We see $T^{60} \in \ker \rho_{5,1,\lambda}$. Since $\rho_{5,1,\lambda}(w)$ is a scalar matrix, $[x,w]$ is also in the kernel of $\rho_{5,1,\lambda}$. We have
\begin{equation*}
\rho_{5,1,\lambda}(xy^{-1}x)^4 = (-iXY^{\inv}X)^4 = I
\end{equation*}
so that $(xy^{\inv}x)^4 \in \ker \rho_{5,1,\lambda}$.  Also
\begin{align*}
\rho_{5,1,\lambda}(xy^{\inv}x)^2 &= (-iXY^{\inv}X)^2 = -I \\
\rho_{5,1,\lambda}(x^{\inv}y)^3 &= (-X^{\inv}Y)^3 = -I \\
\rho_{5,1,\lambda}(x^8 y^{-2})^3 &=  (-X^8Y^{-2})^3 = (e^{2\pi i /3}X^3 Y^3)^3 = (X^3Y^3)^3 = -I
\end{align*}
Further, 
\begin{align*}
\rho_{5,1,\lambda}(zw^{\inv}z)^2  &=  -I \\
\rho_{5,1,\lambda}(p_{60} w^5 zw^{\inv}z)^3  &= (e^{2 \pi i /3} \rho_{5,1,\lambda}(w^5 z w^{\inv}z))^3 = \rho_{5,1,\lambda}(w^{12} z^6) = -I
\end{align*}
Lastly, $\rho_{5,1,\lambda}(w)^{25} = \rho_{5,1,\lambda}(w)$, and so we see that $\Gamma(60) \subseteq \ker \rho_{5,1,\lambda}$. 

If $\lambda= e^{59 \pi i /30}$ then $N=60$ and $X^5 = e^{11\pi i /6}  I$. Here $c=16$ and $d=5$ so that 
\begin{align*}
\rho_{5,1,\lambda}(x)  &=  X^{16} =- iX  \quad & \rho_{5,1,\lambda}(z)  = X^5 = e^{11 \pi i /6}I \\
\rho_{5,1,\lambda}(y) &= Y^{16} = iY \quad & \rho_{5,1,\lambda}(w)  = Y^5 = e^{\pi i /6}I
\end{align*}
and $\rho_{5,1,\lambda}(p_{60}) = \rho_{5,1,\lambda}(z^{20} w z^{-4}  w^{-1}) = \rho_{5,1,\lambda}(z)^4 = e^{4 \pi i /3} I$. We see $T^{60} \in \ker \rho_{5,1,\lambda}$. Since $\rho_{5,1,\lambda}(w)$ is a scalar matrix, $[x,w]$ is also in the kernel of $\rho_{5,1,\lambda}$. We have
\begin{equation*}
\rho_{5,1,\lambda}(xy^{-1}x)^4 = (iXY^{\inv}X)^4 = I
\end{equation*}
so that $(xy^{\inv}x)^4 \in \ker \rho$.  Also
\begin{align*}
\rho_{5,1,\lambda}(xy^{\inv}x)^2 &= (iXY^{\inv}X)^2 = -I \\
\rho_{5,1,\lambda}(x^{\inv}y)^3 &= (-X^{\inv}Y)^3 = -I \\
\rho_{5,1,\lambda}(x^8 y^{-2})^3 &=  (-X^8Y^{-2})^3 = (e^{2\pi i /3}X^3 Y^3)^3 = (X^3Y^3)^3 = -I
\end{align*}
Further, 
\begin{align*}
\rho_{5,1,\lambda}(zw^{\inv}z)^2  &=  -I \\
\rho_{5,1,\lambda}(p_{60} w^5 zw^{\inv}z)^3  &= (e^{4 \pi i /3} \rho_{5,1,\lambda}(w^5 z w^{\inv} z))^3 = \rho_{5,1,\lambda}(w^{12}z^6) = -I
\end{align*}
Lastly, $\rho_{5,1,\lambda}(w)^{25} = \rho_{5,1,\lambda}(w)$, and so we see that $\Gamma(60) \subseteq \ker \rho_{5,1,\lambda}$. The case of $\xi_5 = e^{4 \pi i /5}$ is completely analogous.

\end{proof}

We now move on studying the congruence properties of three-dimensional representations.

\begin{proposition}{\label{maingamma3}} Suppose $\rho: \Gamma \to \mathrm{GL}(3,\C)$ is irreducible with finite image and $2 \leq \po(\rho(\sigma_1)) \leq 5$ and let $N = |\rho(T)|$. Then $\ker \rho$ is a congruence subgroup of level $N$. 
\end{proposition}

\begin{proof} We proceed as in the proof of Proposition \ref{maingamma2}. Let $\spec{\rho(T)} = \Set{\lambda_1, \lambda_2, \lambda_3}$. By Theorem \ref{RT}, we can write $\lambda_2  = e^{2 \pi i j / r}\lambda_1$ and $\lambda_3 = e^{2 \pi i k / r}\lambda_1$ where $2 \leq r \leq 5$ and $j,k \in \ints_r^\times$ distinct. If we define $\tilde{\rho} = \rho \circ \pi$ then this is a representation of $B_3$ that factors through $\pi$ and therefore  
\begin{equation}{\label{key3}}
\lambda_1^6 = e^{-4 \pi i (j+k)/r}. 
\end{equation} 
Each solution to (\ref{key3}) gives rise to a representation $\rho_{r,j,k,\lambda}$ where $\spec(\rho(T)) = \set{ \lambda, e^{2 \pi i j/r}\lambda, e^{2 \pi i k/r}\lambda}$. We finish by showing that the result holds for each of these representations. Again let $\rho(T) = X$ and $\rho(U) = Y$ and note that there are symmetries among the $\rho_{r,j,k,\lambda}$ corresponding to the elements of $S_3$. \\

{\bf{\underline{$r=2$}:}} There are no cases since the eigenvalues cannot be distinct. \\

{\bf{\underline{$r=3$}:}} We must have $\spec \rho(T) = \Set{\lambda_1, e^{2\pi i /3}\lambda_1, e^{4 \pi i /3}\lambda_1}$ where $\lambda_1^6 = 1$. If $\lambda_1 = 1, \,  e^{2 \pi i /3},$ or $e^{4 \pi i /3}$ then $N=3$ so $T^3 \in \ker \rho$.   Also $\rho(U^2T^{-2})^3 = (Y^{\inv}X)^3 = I$ so $\Gamma(3) \subseteq \ker \rho$. If $\lambda_1 = e^{2 \pi i /6}, \, -1,$ or $e^{10 \pi i/6}$ then $N=6$ and $X^3 = -I$. Here $c=4$ and $d=3$ so that 
\begin{align*}
\rho_{3,1,2, \lambda}(x)  &=  X^4 = -X  \quad &\rho_{3,1,2, \lambda}(z)  = X^3 = -I \\
\rho_{3,1,2, \lambda}(y) &= Y^4 = -Y \quad &\rho_{3,1,2, \lambda}(w)  = Y^3 = -I 
\end{align*}
and $\rho_{3,1,2, \lambda}(p_6) = \rho_{3,1,2, \lambda}(z^{20} w z^{-4}  w^{-1}) = I$. We see $T^6 \in \ker \rho_{3,1,2, \lambda}$. Since $\rho_{3,1,2, \lambda}(w)$ is a scalar matrix, $[x,w]$ is also in the kernel of $\rho_{3,1,2, \lambda}$. We have
\begin{equation*}
\rho_{3,1,2, \lambda}(xy^{-1}x)^4 = (-XY^{\inv}X)^4 = I
\end{equation*}
so that $(xy^{\inv}x)^4 \in \ker \rho_{3,1,2,\lambda}$.  Also
\begin{align*}
\rho_{3,1,2, \lambda}(xy^{\inv}x)^2 &= (-XY^{\inv}X)^2 = I\\
\rho_{3,1,2, \lambda}(x^{\inv}y)^3 &= (XY^{\inv})^3 = I \\
\rho_{3,1,2, \lambda}(x^2 y^{-2})^3 &= (X^2Y^{-2})^3 = (Y^{\inv}X)^{-3} = I. 
\end{align*}
Further, 
\begin{align*}
\rho_{3,1,2, \lambda}(zw^{\inv}z)^2  &= I \\
\rho_{3,1,2, \lambda}(p_6 w^5 zw^{\inv}z)^3  &= \rho_{e^{2 \pi i /3},e^{4\pi i 3}, \lambda}(wzw^{\inv}z)^3 = I.
\end{align*}
Lastly, $\rho_{3,1,2, \lambda}(w^{25}) = \rho_{3,1,2, \lambda}(w)$, and so we see that $\Gamma(6) \subseteq \ker \rho_{3,1,2, \lambda}$. \\

{\bf{\underline{$r=4$}:}} The only primitive  $\spec\rho_{4,1,3,\lambda}(T) = \Set{\lambda, i\lambda, -i \lambda}$ where again $\lambda^6 = 1$. When $\lambda = 1$ or $-1$ we have $N=4$. Here $P_4 = T^{20} U T^{-4} U^{-1}$ and so $\rho_{4,1,3,\lambda}(P_4) = X^{20} Y X^{-4} Y^{-1} = I$. In turn, 
\begin{align*}
\rho_{4,1,3,\lambda}(T^4) &= X^4 = I\\
\rho_{4,1,3,\lambda}(P_4 U^5TU^{-1}T)^3 &= ( Y^5 X Y^{\inv} X)^3 = (YXY^{\inv}X)^3 = (XY^{\inv})^3 = I
\end{align*}
which shows $\Gamma(4) \subseteq \ker \rho_{4,1,3,\lambda}$. If $\lambda = e^{2 \pi i /6}$ or $e^{8 \pi i /6}$ then $N=12$ and $X^4 = e^{4 \pi i /3 }I$. Here $c=4$ and $d=9$ so that 
\begin{align*}
\rho_{4,1,3,\lambda}(x)  &=  X^4 = e^{4\pi i /3}I  \quad &\rho_{4,1,3,\lambda}(z)  = X^9 = e^{2 \pi i /3} X \\
\rho_{4,1,3,\lambda}(y) &= Y^4 = e^{2 \pi i 3} I \quad &\rho_{4,1,3,\lambda}(w)  = Y^9 = e^{4 \pi i 3} Y
\end{align*}
and 
\begin{equation*}
\rho_{4,1,3,\lambda}(p_{12}) = \rho_{4,1,3,\lambda}(z^{20}  w z^{-4} w^{-1}) = I.
\end{equation*}

We see $T^{12} \in \ker \rho_{4,1,3,\lambda}$. Since $\rho_{4,1,3,\lambda}(x)$ is a scalar matrix, $[x,w]$ is also in the kernel of $\rho_{4,1,3,\lambda}$. We have
\begin{equation*}
\rho_{4,1,3,\lambda}(xy^{-1}x) = (e^{4 \pi i /3})^3 I = I
\end{equation*}
so that $(xy^{\inv}x)^4 \in \ker \rho_{4,1,3,\lambda}$.  Also
\begin{align*}
\rho_{4,1,3,\lambda}(xy^{\inv}x)^2 &=  I \\
\rho_{4,1,3,\lambda}(x^{\inv}y)^3 &= (e^{2\pi i /3}I)^3 = I \\
\rho_{4,1,3,\lambda}(x^2 y^{-2})^3 &= ( e^{4 \pi i /3} I)^3 = I .
\end{align*}
Further, 
\begin{equation*}
\rho_{4,1,3,\lambda}(zw^{\inv}z) = XY^{\inv}X
\end{equation*}
so that
\begin{equation*}
\rho_{4,1,3,\lambda}(zw^{\inv}z)^2  = I 
\end{equation*} and
\begin{equation*} 
\rho_{4,1,3,\lambda}(p_{12} w^5 zw^{\inv}z)^3  =  (\rho(wz w^{\inv} z))^3 = (e^{4 \pi i /3}YXY^{\inv}X)^3 = (XY^{\inv})^3 = I.
\end{equation*}
Lastly, $\rho_{4,1,3,\lambda}(w)^{25} = \rho_{4,1,3,\lambda}(w)$, and so we see that $\Gamma(12) \subseteq \ker \rho_{4,1,3,\lambda}$. If $\lambda = e^{4 \pi i /6}$ or $e^{10 \pi i /6}$ then $N=12$ and $X^4 = e^{2 \pi i /3 }I$. Here $c=4$ and $d=9$ so that 
\begin{align*}
\rho_{4,1,3,\lambda}(x)  &=  X^4 = e^{2\pi i /3}I  \quad &\rho_{4,1,3,\lambda}(z)  = X^9 = e^{4 \pi i /3} X \\
\rho_{4,1,3,\lambda}(y) &= Y^4 = e^{4 \pi i 3} I \quad &\rho_{4,1,3,\lambda}(w)  = Y^9 = e^{2 \pi i 3} Y
\end{align*}
and 
\begin{equation*}
\rho_{4,1,3,\lambda}(p_{12}) = \rho_{4,1,3,\lambda}(z w z^{-4} w^{-1}) = I.
\end{equation*}

We see $T^{12} \in \ker \rho$. Since $\rho_{4,1,3,\lambda}(x)$ is a scalar matrix, $[x,w]$ is also in the kernel of $\rho_{4,1,3,\lambda}$. We have
\begin{equation*}
\rho_{4,1,3,\lambda}(xy^{-1}x) = (e^{2 \pi i /3})^3 I = I
\end{equation*}
so that $(xy^{\inv}x)^4 \in \ker \rho_{4,1,3,\lambda}$.  Also
\begin{align*}
\rho_{4,1,3,\lambda}(xy^{\inv}x)^2 &=  I \\
\rho_{4,1,3,\lambda}(x^{\inv}y)^3 &= (e^{4\pi i /3}I)^3 = I \\
\rho_{4,1,3,\lambda}(x^2 y^{-2})^3 &= ( e^{2 \pi i /3} I)^3 = I .
\end{align*}
Further, 
\begin{equation*}
\rho_{4,1,3,\lambda}(zw^{\inv}z) = XY^{\inv}X
\end{equation*}
so that
\begin{equation*}
\rho_{4,1,3,\lambda}(zw^{\inv}z)^2  = I 
\end{equation*} and
\begin{equation*} 
\rho_{4,1,3,\lambda}(p_{12} w^5 zw^{\inv}z)^3  =  (\rho_{4,1,3,\lambda}(w z w^{\inv} z))^3 = (e^{2 \pi i /3}YXY^{\inv}X)^3 = (XY^{\inv})^3 = I.
\end{equation*}
Lastly, $\rho_{4,1,3,\lambda}(w)^{25} = \rho_{4,1,3,\lambda}(w)$, and so we see that $\Gamma(12) \subseteq \ker \rho_{4,1,3,\lambda}$. \\

{\bf{\underline{$r=5$}:}} We can always reduce to the cases $\rho_{5,1,2,\lambda}$ with $\lambda^6 = e^{8 \pi i /5}$. Let us consider the first case. We find that $\lambda^6 = e^{8 \pi i /5}$. If $\lambda_= e^{4 \pi i /15}$ or $e^{28 \pi i /5}$ then $N=15$ so $T^{15} \in \ker \rho_{5,1,2,\lambda}$. Also $\rho_{5,1,2,\lambda}(U^2 T^{-8})^3 = (Y^2 X^2)^3 = I$ so $\Gamma(15) \subseteq \ker \rho_{5,1,2,\lambda}$. If $\lambda_1 = e^{19 \pi i /15}$ (resp. $e^{29 \pi i /15}$) then $N=30$ and $X^5 = e^{\pi i /3} I$ (resp. $e^{5 \pi i /3})$. Then $c=16$ and $d=15$ so that 
\begin{align*}
\rho_{5,1,2,\lambda}(x)  &=  X^{16} = -X  \quad &\rho_{5,1,2,\lambda}(z)  = X^{15} = -I \\
\rho_{5,1,2,\lambda}(y) &= Y^{16} = -Y \quad &\rho_{5,1,2,\lambda}(w)  = Y^{15} = -I
\end{align*}
and 
\begin{equation*}
\rho_{5,1,2,\lambda}(p_{30}) = \rho_{5,1,2,\lambda}(z^{20}  w z^{-4} w^{-1}) = I.
\end{equation*}

We see $T^{30} \in \ker \rho_{5,1,2,\lambda}$. Since $\rho_{5,1,2,\lambda}(w)$ is a scalar matrix, $[x,w]$ is also in the kernel of $\rho_{5,1,2,\lambda}$. We have
\begin{equation*}
\rho_{5,1,2,\lambda}(xy^{-1}x) = -XY^{\inv}X
\end{equation*}
so that $(xy^{\inv}x)^4 \in \ker \rho_{5,1,2,\lambda}$.  Also
\begin{align*}
\rho_{5,1,2,\lambda}(xy^{\inv}x)^2 &=  (-XY^{\inv}X)^2 = I \\
\rho_{5,1,2,\lambda}(x^{\inv}y)^3 &= (X^{\inv}Y)^3 = I \\
\rho_{5,1,2,\lambda}(x^8 y^{-2})^3 &= (X^8Y^{-2})^3 = (X^3Y^3)^3 = I .
\end{align*}
Further, 
\begin{equation*}
\rho_{5,1,2,\lambda}(zw^{\inv}z) = -I
\end{equation*}
so that
\begin{equation*}
\rho_{5,1,2,\lambda}(zw^{\inv}z)^2  = I 
\end{equation*} and
\begin{equation*} 
\rho_{5,1,2,\lambda}(p_{30} w^5 zw^{\inv}z)^3  =  (\rho_{5,1,2,\lambda}(w z w^{\inv}  z))^3  = I.
\end{equation*}
Lastly, $\rho_{5,1,2,\lambda}(w)^{25} = \rho(w)$, and so we see that $\Gamma(30) \subseteq \ker \rho_{5,1,2,\lambda}$. If $\lambda = e^{8 \pi i /5}$ then $N=5$ so $T^5 \in \ker \rho_{5,1,2,\lambda}$. Also $\rho_{5,1,2,\lambda}(U^2 T^{-3})^3 = (Y^2 X^2)^3 = I$ so $\Gamma(5) \ker \rho_{5,1,2,\lambda}$. If $\lambda = e^{3 \pi i /5}$ then $N=10$ and $X^5 = -I$. Then $c=6$ and $d=5$ so that 
\begin{align*}
\rho_{5,1,2,\lambda}(x)  &=  X^{6} = -X  \quad &\rho_{5,1,2,\lambda}(z)  = X^{5} = -I \\
\rho_{5,1,2,\lambda}(y) &= Y^{6} = -Y\quad &\rho(_{5,1,2,\lambda}w)  = Y^{5} = -I
\end{align*}
and 
\begin{equation*}
\rho_{5,1,2,\lambda}(p_{10}) = \rho(z^{20} w z^{-4} w^{-1}) = I.
\end{equation*}

We see $T^{10} \in \ker \rho$. Since $\rho_{5,1,2,\lambda}(w)$ is a scalar matrix, $[x,w]$ is also in the kernel of $\rho_{5,1,2,\lambda}$. We have
\begin{equation*}
\rho_{5,1,2,\lambda}(xy^{-1}x) = -XY^{\inv}X
\end{equation*}
so that $(xy^{\inv}x)^4 \in \ker \rho_{5,1,2,\lambda}$.  Also
\begin{align*}
\rho_{5,1,2,\lambda}(xy^{\inv}x)^2 &=  (-XY^{\inv}X)^2 = I \\
\rho_{5,1,2,\lambda}(x^{\inv}y)^3 &= (X^{\inv}Y)^3 = I \\
\rho_{5,1,2,\lambda}(x^3 y^{-2})^3 &= (X^3Y^{-2})^3 = (X^3Y^3)^3 = I .
\end{align*}
Further, 
\begin{equation*}
\rho_{5,1,2,\lambda}(zw^{\inv}z) = -I
\end{equation*}
so that
\begin{equation*}
\rho_{5,1,2,\lambda}(zw^{\inv}z)^2  = I 
\end{equation*} and
\begin{equation*} 
\rho_{5,1,2,\lambda}(p_{30} w^5 zw^{\inv}z)^3  =  (\rho(w z w^{\inv} z))^3  = I.
\end{equation*}
Lastly, $\rho_{5,1,2,\lambda}(w)^{25} = \rho_{5,1,2,\lambda}(w)$, and so we see that $\Gamma(10) \subseteq \ker \rho_{5,1,2,\lambda}$. The other case is analogous.

\end{proof}

\begin{definition} Let $\rho:G \to \mathrm{GL}(d,\C)$ be a representation of a group $G$. Another representation $\rho^\prime$ is called a {\bf{scaling}} of $\rho$ there is some scalar $\theta \in \C^\ast$ so that for all $g \in G$ we have $\rho^\prime(g) = \theta \rho(g)$. In this case we write $\rho^\prime = \theta \rho$ and we say a representation is {\bf{essentially finite}} if it has a scaling with finite image. Clearly if $\rho$ is irreducible then so is $\theta \rho$ for any $\theta$. 
\end{definition}
Note that choices of scalings of a representation are in one-to-one correspondence with elements of the character group of $G$. We are now able to prove Theorem A as a corollary to the above results.

\begin{customthm}{A}{\label{mainbraids}} Let $d = 2$ or $3$ and suppose $\rho: B_3 \to \mathrm{GL}(d,\C)$ is an irreducible essentially finite representation such that $2 \leq \po(\rho(\sigma_1)) \leq5$. Then there is a scaling $\tilde{\rho}$ of $\rho$ so that $\tilde{\rho}$ factors through $\pi$ and if we let $N = | \tilde{\rho}(\sigma_1)|$, then $\pi(\ker \tilde{\rho})$ is a congruence subgroup of level $N$.
\end{customthm}

\begin{proof} First, let $\rho^\prime$ be a scaling of $\rho$ with finite image. Then under $\rho^{\prime}$, we know $(\sigma_1\sigma_2)^3$ acts by some root of unity $\lambda$.  Let $\theta$ be a root of the polynomial $x^6 - \overline{\lambda}$. Define $\tilde{\rho} = \theta\rho$ so that $(\sigma_1\sigma_2)^3$ is in the kernel of $\tilde{\rho}$. There is then a representation $\overline{\rho}$ of $\Gamma$ with finite image so that $\tilde{\rho} = \overline{\rho} \circ \pi $. Since $\po(\rho(\sigma_1)) = \po(\tilde{\rho}(\sigma_1)) = \po(\overline{\rho}(T))$ and $\pi(\ker\rho) = \ker (\rho \circ \pi)$, the result now follows from Proposition  \ref{maingamma2} and Propsition \ref{maingamma3}. 

\end{proof}

\section{Non-congruence kernels} We have shown that for three-dimensional representations, given the projective order of the image of the braid group generator $\sigma_1$ is between two and five, we can ensure that the induced kernel is a congruence subgroup . In this section, we establish that in general this cannot be determined from the projective order alone. Let us begin with a lemma.

\begin{lemma}{\label{NClemma}}Let $\ell > 1$ be an odd integer and let
\begin{equation}{\label{NC}}
X_{\pm} = \begin{pmatrix} e^{\nicefrac{2 \pi i} {3 \ell}} &  \mp e^{\nicefrac{2 \pi i}{3 \ell}} - e^{\nicefrac{-4 \pi i}{3\ell}} & - e^{\nicefrac{2 \pi i}{3 \ell}} \\ 0 & - e^{\nicefrac{2 \pi i}{3\ell}} & - e^{\nicefrac{2 \pi i}{ 3 \ell}} \\ 0 & 0 & e^{\nicefrac{-4 \pi i}{3 \ell} }\end{pmatrix} \quad Y_{\pm} = \begin{pmatrix} \pm e^{\nicefrac{4 \pi i}{3 \ell}} & 0 & 0 \\ \pm e^{\nicefrac{4 \pi i }{3 \ell}} & - e^{\nicefrac{-2 \pi i }{3\ell}} & 0 \\ -e^{\nicefrac{-2 \pi i}{3\ell}} & \pm e^{\nicefrac{-8 \pi i}{3 \ell}} + e^{\nicefrac{-2 \pi i }{3 \ell}} & e^{\nicefrac{-2 \pi i}{3\ell}} 
\end{pmatrix}.  
\end{equation}
Then $[X_{\pm}^{3\ell+1}, Y_{\pm}^{3\ell}]$ has $e^{-2 \pi i /\ell}$ as an eigenvalue. In particular, $[X_{\pm}^{3\ell+1},Y_{\pm}^{3\ell}]$ is not the identity matrix. 
\end{lemma}
\begin{proof} Let $M_{\pm} = [X_{\pm}^{3\ell+1},Y_{\pm}^{3\ell}]$. First note that 
\begin{align*}
M_{\pm} &= X_{\pm}^{3\ell+1}Y_{\pm}^{3\ell} X_{\pm}^{-(3\ell+1)} Y_{\pm}^{-3\ell} \\
    &= X_{\pm}^{3\ell+1} Y_{\pm}^{3\ell-1} (Y_{\pm}X_{\pm}^{-(3\ell+1)}Y_{\pm}^{-1})Y_{\pm}^{-(3\ell-1)}.
\end{align*}
We will show that each of $X_{\pm}^{3\ell+1}, \, Y_{\pm}^{3\ell-1}, \, Y_{\pm}X_{\pm}^{-(3\ell+1)}Y_{\pm}^{-1}, \, \text{and, } Y_{\pm}^{-(3\ell-1)}$ have $v=(0,0,1)$ as a right eigenvector so that their product $M_{\pm}$ does as well. 

Now since $X_{\pm}$ is upper triangular, we know $v$ is already a right eigenvector of $X_{\pm}^{3\ell+1}$ with eigenvalue $e^{-4(3\ell+1)\pi i /3\ell} = e^{-4\pi i /3 \ell}$. We can compute that 
\begin{equation*}
Y_{\pm}^2 =\begin{pmatrix} e^{\nicefrac{8\pi i}{\ell}} & 0 & 0 \\ e^{\nicefrac{8 \pi i }{\ell}} \mp e^{\nicefrac{2\pi i}{3\ell}} & e^{\nicefrac{-4 \pi i}{\ell}} & 0 \\ 0 & 0 & e^{\nicefrac{-4 \pi i}{\ell}} \end{pmatrix}
\end{equation*}
and since $3 \ell-1$ is even, we see that both $Y_{\pm}^{3 \ell-1}$ and its inverse both have $v$ as a right eigenvector. Another computation shows
\begin{equation*}
Y_{\pm}X_{\pm}^2Y_{\pm}^{-1} = \begin{pmatrix} e^{\nicefrac{4\pi i}{3\ell}} & 0 & 0  \\  e^{\nicefrac{4 \pi i}{3\ell}}  \mp e^{\nicefrac{-2\pi i }{3\ell}} & e^{\nicefrac{-8 \pi i}{3\ell}} & \mp e^{\nicefrac{4 \pi i}{3 \ell}} \pm e^{\nicefrac{-2 \pi i }{3\ell}}   \\ 0 & 0 & e^{\nicefrac{4 \pi i}{3 \ell}} \end{pmatrix}
\end{equation*}
Again $3\ell+1$ is even so $Y_{\pm}X_{\pm}^{-(3\ell+1)} Y_{\pm}^{\inv} = [(Y_{\pm}X_{\pm}^2Y_{\pm}^{-1})^{-1}]^{\nicefrac{3\ell+1}{2}}$ has $v$ as a right eigenvector with eigenvalue $(e^{4\pi i /3\ell})^{-(3\ell+1)/2} = e^{-2 \pi i /3\ell}$. Thus $M_{\pm}$ too has $v$ as a right eigenvector and the corresponding eigenvalue is $e^{-4\pi i /3 \ell} e^{-2 \pi i /3 \ell} = e^{-2 \pi i /\ell}$.

\end{proof}

\begin{customthm}{B} For each odd positive integer $\ell$ greater than one, there are two irreducible representation $\rho_{\ell,\pm}$ with finite image so that $\rho_{\ell,\pm}(\sigma_1)$ has projective order $2 \ell$ and $\pi(\ker\rho_{\ell,\pm})$ is a non-congruence subgroup of $\Gamma$ with geometric level $6\ell$. 
\end{customthm}

\begin{proof} Fix $\ell$ satisfying the hypotheses above and let $\rho_{\ell,\pm}: B_3 \to \mathrm{GL}(3,\C)$ be given by $\rho_{\ell,\pm}(\sigma_1)=X_\pm$ and $\rho_{\ell,\pm}(\sigma_2) = Y_{\pm}^{-1}$ with $X_{\pm}$ and $Y_{\pm}$ as above. Since $(e^{2 \pi i / 3 \ell} \cdot -e^{2 \pi i /3 \ell} \cdot \pm e^{-4 \pi i / 3\ell})^2 = 1$, we see by Theorem \ref{TW} that $\rho_{\ell,\pm}$ factor through $\pi$. Let $\overline{\rho}_{\ell,\pm}$ be the induced representations. Rowell and Tuba's criteria in dimension three implies $\rho_{\ell,\pm}$ (hence $\overline{\rho}_{\ell,\pm}$) has finite image since $\spec(\rho_{\ell,\pm}(\sigma_1)) = \set{\pm e^{2 \pi i/3\ell}, \pm e^{-4 \pi i /3\ell}}$. Its easy to compute $\po(X_\pm) = 2\ell$ and that $X_\pm$ has order $6 \ell$. Therefore, if we let $K_{\ell,\pm}$ equal $\pi (\ker \rho_{\ell,\pm})$ then each $K_{\ell,\pm}$ is a finite index subgroups of $\Gamma$ with geometric level $6\ell$. We are done if $\Gamma(6 \ell)$ is not a subgroup of $K_{\ell,\pm}$. 

Taking $G_{6 \ell}$ as in Proposition \ref{Hsu}, we see $K_{\ell,\pm}$ is non-congruence if and only if for some element of $\gamma \in G_{6 \ell}$ we have $\overline{\rho}_{\pm}(\gamma)$ is not the identity matrix. Observe that since $\ell$ is odd, the values of $c$ and $d$ in Proposition \ref{Hsu} can be determined to be $c=3\ell+1$ and $d=3\ell$. Therefore $[T^{3\ell+1}, U^{3\ell}] \in G_{6\ell}$. But then 
\begin{align*}
\overline{\rho}_{\ell,\pm}\left( [T^{3\ell+1}, U^{3\ell}]\right) &= \overline{\rho}_{\ell,\pm}\left( [\pi(\sigma_1)^{3\ell+1} , \pi(\sigma_2^{-1})^{3\ell}]\right)\\
									&= \rho_{\ell,\pm} \left( [\sigma_1^{3\ell+1}, \sigma_2^{-3\ell}] \right) \\
									&= [X_{\pm}^{3\ell+1},Y_{\pm}^{3\ell}]\\
									&= M_{\pm}
\end{align*}
so by Lemma \ref{NClemma} $[T^{3\ell+1}, U^{3\ell}]$ is not in the kernel of $\rho_{\ell,\pm}$. Thus, we have shown each $K_{\ell,\pm}$ is non-congruence. 
\end{proof}

\section{Realization by modular tensor categories}

Our results in section 3 apply to arbitrary two-dimensional representations. Let us now turn our attentions towards those arising from modular tensor categories. We will show that there are quantum representations of $B_3$ satisfying Corollary \ref{mainbraids} for $r=2,3,4$. For each representations $\rho$, it will be enough to know the eigenvalues of $\rho(\sigma_1)$. 

For $r=2,3,$ and $4$ there are two-dimensional irreducible unitary representations $\rho$  with $\po(\rho(\sigma_1)) = r$ arising from weakly integral modular tensor categories. The modular tensor category structure of $\mathcal{C}=\Rep D(S_3)$ was described completely in \cite{Cui}. The authors label the simple objects of this category with the letters $A$ through $G$. It was shown that the anyon $C$ corresponding to the standard representation of the $S_3$ gives rise to a unitary representation $\tilde{\rho}_C: B_3 \to \mathrm{U}(\Hom(C, C^{\otimes 3}))$. This representation has a two-dimensional irreducible summand $\rho_C$ so that with respect to some basis
\begin{equation*}
\rho_C(\sigma_1) =  \begin{pmatrix} i & 0 \\ 0 & -i \end{pmatrix}.
\end{equation*}
One can see that $\po(\rho_C(\sigma_1)) = 2$.  Also arising from $\Rep D(S_3)$ is a unitary representation $\tilde{\rho}_D: B_4 \to \mathrm{U}(\Hom(B, D^{\otimes 4}))$. However, the vector space $\Hom(B,D^{\otimes 4})$ is isomorphic to $\Hom(B \otimes D, D^{\otimes 3})$ which gives a representation of $B_3$ with a two-dimensional irreducible summand $\rho_D$. In this case, we can choose a basis so that
\begin{equation*}
\rho_D(\sigma_1) = \begin{pmatrix} \omega  &  0 \\  0 & \omega^2 \end{pmatrix}
\end{equation*}
where $\omega = e^{2 \pi i/3}$ and  so $\po(\rho_D(\sigma_1)) = 3$. The Ising category is a rank 3 modular tensor category with simple objects usually denoted $\Set{ 1 , \sigma, \psi}$ in the literature and is closely related to the Chern-Simmons-Witten $\mathrm{SU}(2)$-TQFT at level 2 or equivalently the purification of $\Rep U_q\mathfrak{sl}_2 \C $ where $q$ is an appropriate choice of 16th root of unity. It arises as a Temperley-Lieb-Jones algebroid with Kauffman variable $A = i e^{-\pi i /16}$ (see \cite{Wang}). The simple object $\sigma$ 
affords a two-dimensional irreducible representation $\rho_\sigma: B_3 \to \mathrm{U}(\Hom(\sigma, \sigma^{\otimes 3}))$ where in the canonical basis of admissibly labelled trivalent trees
\begin{equation*}
\rho_\sigma(\sigma_1)= \begin{pmatrix} e^{-\pi i /8} & 0 \\ 0 & e^{3\pi i/8} \end{pmatrix}
\end{equation*}
thus $\po(\rho_\sigma(\sigma_1)) = 4$. Although not defined using the braiding on the category, it is worth pointing out that the quantum representation of $\mathrm{SL}(2,\ints)$ associated to the Fibonacci category maps $T$ to a matrix with projective order 5 and this pulls back to a representation $\rho$ of $B_3$ so that $\po(\rho(\sigma_1)) = 5$.

There is a three-dimensional irreducible representation of $B_3$ arising from $\Rep D(S_3)$ with finite image, none of whose scalings project onto a congruence subgroup of $\Gamma$. The anyon $G$ in this category leads us to an irreducible representation $\rho_G: B_3 \to \mathrm{U}(\Hom(G, G^{\otimes 3}))$ with finite image. In the basis of the admissibly labeled trivalent trees, this representation is given by 
\begin{equation*}
\rho_G(\sigma_1) = \omega^2 \begin{pmatrix} 1 & 0 & 0 \\ 0 & -1 & 0 \\ 0 & 0 & \omega^2 \end{pmatrix} \qquad \rho_G(\sigma_2) = \omega \begin{pmatrix} \frac{1}{2} & -\frac{1}{2} & \frac{1}{\sqrt{2}} \omega \\ - \frac{1}{2} & \frac{1}{2} & \frac{1}{\sqrt{2}}\omega \\ \frac{1}{\sqrt{2}} \omega & \frac{1}{\sqrt{2}} \omega & 0 \end{pmatrix} 
\end{equation*}
where $\omega = e^{2\pi i /3}$. The generator of the center acts by $\omega$. If we let $\tau=e^{-\pi i / 9 }$ and $\rho= \tau \rho_G$ then $\spec(\rho(\sigma_1)) = \set{\pm e^{2 \pi i/9}, e^{5 \pi i/9} }$ and so $\rho$ is conjugate to $\rho_{3,-}$. In particular, the induced kernel $\pi(\ker \rho)$ is a non-congruence subgroup of $\Gamma$.

\bibliography{biblio} 
\bibliographystyle{bib}

\end{document}